\documentclass[a4paper]{amsart}
\usepackage{amsmath}
\usepackage{amssymb}
\usepackage{amsfonts}
\usepackage{amsthm}
\usepackage{amssymb}
\usepackage{amsbsy}
\usepackage{color}
\usepackage{graphicx}
\usepackage{enumerate}
\usepackage{enumitem}
\usepackage{subfigure}
\usepackage[utf8]{inputenc}
\usepackage{csquotes}
\usepackage[greek,english]{babel}
\usepackage[hyperref,doi,url=true,natbib,backend=biber,giveninits=true]{biblatex}
\usepackage{hyperref}

\AtEveryBibitem{
 \clearlist{address}
 \clearfield{date}
 \clearfield{eprint}
 \clearfield{isbn}
 \clearfield{issn}
 \clearlist{location}
 \clearfield{month}
 \clearfield{series}
 \clearlist{language}
 \clearfield{urldate}
 \ifentrytype{unpublished}{}{
   \clearfield{url}
   \clearfield{urldate}}
 \ifentrytype{book}{}{
   \clearlist{publisher}
   \clearname{editor}
 }
}

\DeclareFieldFormat{urldate}{}

\makeatletter
\newcommand{\leqnomode}{\tagsleft@true}
\newcommand{\reqnomode}{\tagsleft@false}
\makeatother

\newtheorem{theorem}{Theorem}[section]
\newtheorem{corollary}[theorem]{Corollary}
\newtheorem{lemma}[theorem]{Lemma}
\newtheorem{proposition}[theorem]{Proposition}

\theoremstyle{definition}
\newtheorem{definition}[theorem]{Definition}
\theoremstyle{remark}
\newtheorem{remark}[theorem]{\textbf{Remark}}
\newtheorem{example}[theorem]{Example}
\newtheorem{question}[theorem]{Question}
\numberwithin{equation}{section}

\newcommand{\E}{\mathbb{E}}

\renewcommand{\P}{\mathbb{P}}

\newcommand{\M}{\mathbb{M}}

\newcommand{\di}{\mathrm{d}}

\newcommand{\pian}[2]{\dfrac{\partial #1}{\partial #2}}

\newcommand{\R}{\mathbb{R}}
\newcommand{\Rp}{\R_+}

\newcommand{\Lc}{\mathcal{L}}
\newcommand{\as}{a.s.}

\newcommand{\ie}{i.e.}

\newcommand{\eg}{e.g.}
\newcommand{\dx}{\mathrm{d}x}
\newcommand{\dy}{\mathrm{d}y}

\newcommand{\dt}{\mathrm{d}t}

\newcommand{\F}{\mathcal{F}}
\newcommand{\Fc}{\mathcal{F}}
\newcommand{\Gc}{\mathcal{G}}

\newcommand{\Bc}{\mathcal{B}}
\newcommand{\eps}{\varepsilon}
\newcommand{\indic}[1]{\boldsymbol{1}_{\{\ensuremath{#1}\}}}

\DeclareMathOperator{\supp}{supp}

\newcommand{\Lip}{\mathcal{L}ip}
\newcommand{\SL}{\mathcal{SL}}

\newcommand{\Ep}[1]{\E\left[#1\right]}

\renewcommand{\Pr}{\P}

\renewcommand{\Mc}{\mathcal{M}}
\newcommand{\Pc}{\mathcal{P}}
\newcommand{\Wc}{\mathcal{W}}
\newcommand{\W}{\mathcal{W}}
\newcommand{\cadlag}{c\`adl\`ag}

\usepackage{xcolor}

\definecolor{orange}{rgb}{1,0.3,0.2}

\newcommand{\nn}{\nonumber}
\newcommand{\Nc}{\mathcal{N}}
\newcommand{\om}{\omega}
\newcommand{\Rc}{\mathcal{R}}

\DeclareMathOperator{\essinf}{ess\;inf}
\title[MVMs and Bass-type solutions]{Measure-valued martingales and optimality of Bass-type solutions to the Skorokhod Embedding Problem}

\author[Beiglb\"ock]{Mathias Beiglb\"ock}
\thanks{MB gratefully acknowledges support through FWF grant Y782.}
\address{
Technische Universit{\"a}t Wien, Vienna, Austria. 
e-mail: \texttt{mathias.beiglboeck@tuwien.ac.at}.
}

\author[Cox]{Alexander M.~G.~Cox}
\address{
University of Bath., Bath, U.~K..
e-mail: \texttt{a.m.g.cox@bath.ac.uk}.
}

\author[Huesmann]{Martin Huesmann}
\thanks{MH gratefully acknowledges support from the Deutsche Forschungsgemeinschaft
through the CRC 1060 The Mathematics of Emergent Effects and the Hausdorff Center for Mathematics.}
\address{
University of Bonn, Bonn, Germany. 
e-mail: \texttt{huesmann@iam.uni-bonn.de}.
}

\author[K\"allblad]{Sigrid K{\"a}llblad}
\address{
Technische Universit{\"a}t Wien, Vienna, Austria. 
e-mail: \texttt{sigrid.kaellblad@tuwien.ac.at}.
}

\date{\today}

\begin{document}
\begin{abstract}
  In this paper we consider (probability-)measure valued processes, which we call MVMs, which have a natural martingale structure. Following previous work of Eldan and Cox-K\"{a}llblad, these processes are known to have a close connection to the solutions to the Skorokhod Embedding Problem. In this paper, we consider properties of these processes, and in particular, we are able to show that the MVMs connected to the Bass and Root embeddings have natural measure-valued analogues which also possess natural optimality properties. We also introduce a new MVM which is a generalisation of both the Bass and Root MVMs.
\end{abstract}
\maketitle

\section{Introduction}
	
We consider here an alternative approach to the Skorokhod embedding problem (SEP) which is based on viewing real-valued processes as measure-valued stochastic processes. The Skorokhod Embedding Problem (SEP) is a longstanding and classical problem in probability; we refer the reader to the survey articles \cite{Ob,Ho}, and the recent paper \cite{BCH} for a more recent approach inspired by methods from Optimal Transport.

The observation which connects the SEP and measure-valued processes is the following: that, modulo technicalities, there is a one-to-one correspondence between:
\begin{itemize}
\item the set of (continuous, UI) martingales $M_t$ with $M_0 = 0$ and $M_\infty \sim \mu$;
\item the set of (continuous) measure-valued martingales (MVMs) $\xi_t$ with $\xi_0 = \mu$ and $\xi_\infty = \delta_y$ for some $y \in \R^d$.
\end{itemize}
The first formulation is (up to a time change, and for 1-dimensional processes), equivalent to the SEP, and we say that a measure-valued process $\xi_t$ is a measure-valued martingale (MVM) if, for any $A \in \Bc(\R^d)$, $\xi_t(A)$ is a martingale. To see the correspondence, given $M_t$, we may define $\xi_t$ by $\xi_t(A):= \Pr(M_\infty \in A | \Fc_t)$ and, conversely, $M_t:= \int x \, \xi_t(\dx)$; see Section~\ref{sec:basic} for details.

While there is a one-to-one correspondence between processes and MVMs on $[0,\infty)$, for any $t<\infty$, the process $\xi_s$, $s\le t$, naturally carries more information than the corresponding process $M_s$, $s\le t$, since also the marginal distribution to eventually be embedded is known. In \cite{CoxKallblad}, this was used to formulate optimal embedding problems as dynamic programming problems, exploiting the fact that the MVM approach allows one to include the terminal embedding constraint in the current state of the controlled process; see also \cite{BCS,BM,SK}. MVMs and their connection to the SEP have also appeared in \cite{Eldan:2016aa}. There is also a stream of literature dating back to at least \cite{Walsh:1986aa} who look at \emph{martingale measures}, where the assumption that the measures are probability measures is typically dropped. Subsequently, a substantial literature on these processes has developed, although much of it under the assumption that the martingale measures are either \emph{orthogonal}, or have \emph{nuclear covariance}, \eg{} \cite{El-Karoui:aa}. Such processes have been useful in a variety of applications; see \eg{} \cite{imkeller2001free}.

Here, we consider an application of the measure-valued viewpoint that truly exploits the metric structure of the underlying space of measures. Specifically, motivated by the theory of optimal transport, we equip the space of measures with the $1$-Wasserstein metric. The metric structure then allows for a study of functions defined on this space.  First, we study the evolution over time of the Wasserstein distance between to measure-valued martingales: we show that for any two MVMs $\eta_t$ and $\xi_t$, the process $t\mapsto \Wc(\xi_t,\eta_t)$ is a sub-martingale for $t\ge 0$.
Second, we introduce a suitable notion of `speed'. That is, a means by which one can measure how fast a measure-valued process evolves over time -- while the speed of an MVM $\xi_t$ is always bounded from below by the quadratic variation of its mean process $M_t$, these quantities will in general differ.

By use of these observations, we then obtain two optimality properties of the well-known Bass \cite{Bass:1983aa} solution to the SEP. We show that when $\eta_t$ is the MVM associated with a Brownian motion, then the above-mentioned process measuring the distance of $\eta_t$ to $\xi_t$, is in fact a martingale -- as opposed to a sub-martingale -- when evaluated for the Bass solution. 
On the other hand, the Bass MVM also miminises the speed among all MVMs. We also obtain an optimality property of the Root \cite{root1969existence} solution. In fact, modulo a suitable scaling of space, the Bass and Root solutions solve the same optimisation problem. 
We emphasise that these optimality properties genuinely exploit the properties of the embeddings viewed as measure-valued martingales rather than stopped processes.  We also discuss Markov properties of the involved MVMs.

\subsection{Notation}

We recall some useful definitions and notation relating to probability measures. Specifically, we recall that the pushforward of a probability measure $\lambda$ on $X$ by a function $f: X \to Y$, is the measure $\mu$ given by $\mu(A) = \lambda(f^{-1}(A))$, and is denoted $f_{\#}\lambda$.  Further, on the space of probability measures, we introduce the functional $\M(\mu) = \int x \, \mu(\dx)$ for the mean.  For a semimartingale $X$, we write $\langle X \rangle$ for the quadratic variation process.

\section{Basics on Measure-Valued Martingales}\label{sec:basic}

For any $p\ge 1$, we introduce the set of $p$-integrable probability measures:
\begin{equation}
  \label{eq:M1Defn}
  \Pc_p:= \left\{ \mu \in \Mc_+(\R^d) : \mu(\R^d) = 1, \int |x|^p \, \mu(\dx) < \infty\right\},
\end{equation}
where $\Mc_+(X)$ is the set of non-negative measures on $X$; we will mostly be interested in the set $\Pc_1$.

We consider a fixed underlying filtered probability space, $(\Omega, \Fc, (\Fc_t), \Pr)$, satisfying the usual conditions.
\begin{definition}
  \label{def:MVM}
  We say that an adapted process $(\xi_t)_{t\geq 0}$ with $\xi_t\in\Pc_1$, $t\ge 0$, is a {\it measure-valued
    martingale} if, for any $f \in C_b(\R^d)$, $\xi_\cdot(f) := \int f(x)
  \xi_\cdot(\dx)$ is a martingale. 
\end{definition}

Note, trivially, that for $f \in C_b(\R^d)$, the martingale $\xi_\cdot(f)$ is uniformly integrable with well defined limit $\xi_{\infty}(f)$ (in particular, $\xi_\infty$ is a measure; see \citep[Proposition~2.1]{Horowitz:aa}).  We also note that an adapted process $\xi_\cdot\in\Pc_1$ is a measure-valued martingale, if and only if, $\xi_\cdot(A)$ is a martingale for any $A\in\Bc(\R)$, and then, in fact, $\xi_\cdot(f)$ is a martingale for any (non-negative) measurable function; see \citep[Remark~2]{CoxKallblad}.

We also wish to discuss continuity of measure-valued martingales. In order to do this, we make the following definition:
\begin{definition}
  \label{def:CMVM}
  We say that a measure-valued martingale is \emph{continuous} if, for any 1-Lipschitz function $f$, $\xi_\cdot(f) = \int f(x) \, \xi_\cdot(\dx)$ is a continuous process.
\end{definition}

It immediately follows that $\M(\xi_\cdot)$ is a continuous process whenever $\xi_\cdot$ is continuous; conversely, whenever $\xi_\cdot$ is continuous in the sense of weak convergence of measures, and also its mean is continuous, it is continuous in the above sense.  Specifically, Definition \ref{def:CMVM} is equivalent to requiring continuity of $t \mapsto \xi_t$ in the topology of $\Wc_1$, the $L^1$ Wasserstein metric (cf. \eqref{eq:wasserstein_metric} below), by the duality of the Wasserstein distance \citep[Theorem 6.1.1]{Ambrosio:2008aa}.
  	
Since our underlying probability space is assumed to satisfy the usual conditions, of course, for every $f \in C_b(\Rp)$, the martingale $\xi_\cdot(f)$ has a \cadlag{} version. More pertinently, according to \citep[Theorem~2.5]{Horowitz:aa} and \citep[Remark~4]{CoxKallblad}, we can choose a version of any given measure-valued martingale $\xi$, such that $\xi_\cdot(f)$ is \cadlag{} for every 1-Lipschitz function $f$.  In what follows, we will assume that we always take this \cadlag{} (in the sense of Definition \ref{def:CMVM}) version of $\xi$.  Moreover, by a slight modification of \citep[Remark~4]{CoxKallblad}, given that $\xi_0\in\Pc_p$, for some $p\ge 1$, there exists a version of $\xi$ which is \cadlag{} with respect to continuity in the topology induced by $\Wc_p$; whenever $\xi_0\in\Pc_p$ we will assume that we take this version.

We need one further concept. Consider the set of singular measures on $\R^d$: $\Pc^s := \{ \mu \in \Pc_1 : \mu = \delta_y, y \in \R^d\}$. Motivated by the fact that the support of a measure-valued martingale can only ever decrease -- if $\xi_{t_0}(A) = 0$, then $\xi_t(A) = 0$ for all $t \ge t_0$ -- we define as follows:

\begin{definition}
  \label{def:TMVM}
  We say that a measure-valued martingale $\xi_\cdot$ is {\it terminating} if $\xi_t \to \xi_\infty\in \Pc^s$ \as{} as $t\to\infty$, where the convergence is in the sense of weak convergence of measures. It is {\it finitely terminating} if $\tau_s:= \inf \{t \ge 0: \xi_t \in \Pc^s\}$ is finite a.s.
\end{definition}

\begin{lemma} \label{lem:MVMProp} Suppose $\xi_\cdot$ is a terminating measure-valued martingale with $\xi_0 = \mu$.  Then $X_\cdot := \M(\xi_\cdot)$ is a UI martingale with $X_{\infty} \sim \mu$.
\end{lemma}
\begin{proof}
  First, observe that by the martingale property, for $f \in C_b(\R^d)$,
  \begin{align*}
    \Ep{f(X_{\infty})} =  \Ep{\int f(x) \, \xi_\infty(\dx)}
                       =  \Ep{\int f(x) \, \xi_0(\dx)}  
                       =  \int f(x) \, \mu(\dx).
  \end{align*}
  and so $X_\infty \sim \mu$. In particular, $X_\infty$ is integrable.

  Now observe that the martingale statement can be shown for each co-ordinate individually. Write $f_K(x):= (x \wedge K)\vee(-K)$. Then $\xi_t \in \Pc_1$ implies that \[\int f_K(x) \, \xi_t(\dx) \to \int x \, \xi_t(\dx) = X_t\] as $K \to \infty$. Moreover, $\int f_K(x) \, \xi_t(\dx) = \Ep{f_K(X_\infty)|\Fc_t} \to \Ep{X_\infty|\Fc_t}$ by the (conditional) Dominated Convergence Theorem. Hence $X_t = \Ep{X_\infty|\Fc_t}$ for all $t \ge 0$, and it follows that $X_t$ is a UI martingale.
\end{proof}

\begin{corollary}
  If $\xi_\cdot$ is a terminating measure-valued martingale with $\xi_0=\mu$, then for every $1$-Lipschitz function $f$, $X^f_\cdot := \xi_\cdot(f)$ is a uniformly integrable martingale with $X_0^f = \int f \, \di \mu$ and $X_\infty^f \sim f_\#\mu$.
\end{corollary}

\section{Examples of MVMs}\label{sec:ex}

In this section we will introduce certain natural MVMs, which will also become an important focus for our later optimality results. We will pay particular attention to MVMs associated with various solutions to the Skorokhod embedding problem. When doing so, we restrict to the case of one dimension, $d=1$, and suppose that the given probability space $(\Omega, \Fc, (\Fc_t), \Pr)$ supports a Brownian motion $(B_t)_{t\ge 0}$.
	
Recall that for a given centred distribution $\mu\in\Pc_1(\R)$, a solution to the SEP is a stopping time $\tau$ such that $B_\tau\sim\mu$ and $(B_{t \wedge \tau})_{t \ge 0}$ is uniformly integrable. Following \cite{Eldan:2016aa,CoxKallblad}, the SEP is equivalent to the problem of finding a terminating MVM with $\xi_0=\mu$ and such that $\M(\xi_t)$ is a Brownian motion up to time $\tau:=\inf\{t\ge 0: \xi_s\in\Pc^s\}$.  Specifically, if we let
\begin{equation} \label{eq:def_sep_mvm}
  \xi_t(A):=\P\left(B_{\tau}\in A|\F_t\right), \quad A\in\Bc(\R),
\end{equation}
then the process $\xi_t$ is an MVM with $\xi_0 = \mu$. Note that by definition $\xi_0=\mu$ and $\xi_t$ is terminating. 

We now consider some examples (see also \cite{Eldan:2016aa}).

\begin{itemize}[wide,itemsep=2ex]
\item {\bf The Root MVM:} Root \cite{root1969existence} (see also \cite{kiefer1972skorohod,rost1976skorokhod}) showed that there exists a barrier $\Rc$ such that $\tau_\Rc:=\inf\{t\ge 0: (t,B_t)\in\Rc\}$ is a solution to the SEP; we denote the associated MVM defined via \eqref{eq:def_sep_mvm} by $\xi^\Rc$.  More precisely, a barrier is a subset of $\R_+ \times \R$ such that $(t,x) \in \Rc$ implies $(s,x) \in \Rc$ for all $s \ge t$. The results of Root (and subsequent authors) then show that for any centred, integrable probability measure $\mu$, there exists a barrier $\Rc$ such that $B_{\tau_\Rc} \sim \mu$, and $B_{t \wedge \tau_{\Rc}}$ is uniformly integrable. Moreover, by \cite{loynes1970stopping}, such a barrier is unique.  Given the barrier $\Rc$, we can further define stopping times $\tau_\Rc^{t,x}:= \inf \{s \ge 0: (t+s,x+W_s) \in \Rc\}$ for some Brownian motion $W$, $W_0 = 0$. Then $\mu_{\Rc}^{t,x}(A) := \Pr(x+W_{\tau_{\Rc}^{t,x}} \in A)$ defines a class of probability measures, and we have $\xi^\Rc_t = \mu_{\Rc}^{t,B_t}$, a.s.  Note that $\xi^\Rc$ is a terminating MVM (since $\tau_\Rc$ is almost surely finite) and that $B_{t\wedge \tau_{\Rc}} = \int x \, \xi^\Rc_t(\dx)$.

\item {\bf The Bass MVM:} A second natural construction of MVMs, which has a strong `transport' influence, and is closely related to the construction of a solution to the SEP due to \cite{Bass:1983aa} (see also \cite{Eldan:2016aa,Hariya:2014aa}), is described as follows. Let $(B_t)_{t \in [0,1]}$ be a Brownian motion started at 0, and observe then that $\eta_t(A) := \Pr(B_1 \in A| \Fc_t)$ is an MVM given by $\eta_t = \Nc(B_t,1-t)$. Let $\mu$ be a given (integrable, centred) probability measure. Write $F_{\mu}, F_{\eta_t}$ for the c.d.f.s of $\mu$ and $\eta_t$ respectively, and write $h = F_{\mu}^{-1} \circ F_{\eta_0}$. We will call $h$ the \emph{scale function} of the MVM. Then $h(B_1) \sim \mu$. Moreover, it is easily seen that the measure-valued process $\xi^h$ defined by $\xi^h_t:=h_{\#} \eta_t$, or equivalently 
  \begin{equation*}
    \int f(x) \, \xi_t^h(\dx) := \int (f\circ h)(x) \, \eta_t(\dx) = \int f(x) \,(h_{\#} \eta_t)(\dx)  = \Ep{f(h(B_1))|\Fc_t}, 
  \end{equation*}
  for $f\in C_b(\R^d)$, defines an MVM with $\xi_0^h = \mu$.

  Strictly speaking, to recover Bass' solution to the SEP, one needs to show how to convert this MVM into a stopping time for a Brownian motion. Since $\xi^h$ is an MVM, then $M^h_\cdot:= \M(\xi^h_\cdot)$ is a (continuous) martingale. In particular, there exists a time-change $\tau_t := \inf \{s \ge 0: \langle M^h\rangle_s \ge t\}$ such that $M_{\tau_t}$ is a Brownian motion (with respect to the filtration $\Gc_t := \Fc_{\tau_t}$) up to the ($\Gc$-)stopping time $\sigma:= \tau_1^{-1} = \sup \{ s \ge 0 : \tau_s < 1\}$, and $B_\sigma \sim \mu$. Bass' solution is then the stopping time $\sigma$ (which can also be constructed directly through stochastic calculus arguments). Of course, the MVM $\tilde{\xi}_t^h := \xi_{\tau_t}^h$ is another example of an MVM. To distinguish, we will call $\xi^h$ the \emph{canonical-time} Bass embedding, and $\tilde{\xi}^h$ the \emph{natural-time} Bass embedding. We will typically be interested in the canonical-time MVM, and we will often call this MVM simply the Bass MVM.

  Note that in both cases, the `canonical' choice of $h = F_{\mu}^{-1} \circ F_{\eta_0}$ is not necessary to deduce that the resulting process is an MVM. In fact, an arbitrary $h$ such that $h_{\#} \eta_0 = \mu$ can be chosen to determine an MVM.

\item {\bf The Az\'ema-Yor MVM:} The Az\'{e}ma-Yor solution to the SEP, \cite{azema1979solution}, is given by $\tau_{AY}:= \inf\{ t \ge 0: B_t \le \psi(S_t)\}$, where $\psi$ is the inverse Barycentre function of $\mu$ and $S_t:=\sup_{s\le t} B_t$. To be specific, we introduce the following notation from \cite{Cox:aa}: given a (probability) measure $\mu$, and $p \in [0,1]$, we write $\mu_p$ for the (sub-probability) measure given by $\mu_p((-\infty,x]) = (\mu((-\infty,x])-p)_+$. Then it is clear that $p \in [0,1) \mapsto \frac{1}{p}\int x \, \mu_p(\dx)$ is a continuous, increasing function, and in fact strictly increasing as long as $p \le \mu(\{\sup \supp \mu\})$. In particular, writing $\sup \supp \mu = r$, and $\inf \supp \mu = \ell$ we see that for all $x \in \left[\int x \, \mu(\dx), r\right)$ there exists $\pi(x)\in [0,1)$ such that $x (1-\pi(x)) = \int x \, \mu_p(\dx)$. The inverse barycentre function for $\mu$, $\psi$, is then given by $\psi(x) = F_{\mu}^{-1}(\pi(x))$ for $x < r$, and $\psi(r) = r$. Moreover, the corresponding MVM can then be written as:
  \begin{align*}
    \xi^{AY}_t
    & =
      \begin{cases}
        \frac{S_t - B_t}{S_t - \psi(S_t)}\delta_{\psi(S_t)} + \frac{B_t-\psi(S_t)}{S_t - \psi(S_t)} \frac{\mu_{\pi(S_t)}}{1-\pi(S_t),}& \quad t < \tau_{AY}\\
        \delta_{\psi(S_{\tau_{AY}})},& \quad  t \ge \tau_{AY}.
      \end{cases}
  \end{align*}

\item {\bf The Bass-Root MVM:} The final example we give is a combination of the Bass and Root cases. Specifically, let $\Rc$ be the Root barrier associated with a non-atomic measure $\lambda$. Given a measure $\mu$, choose the function $\kappa = F_{\mu}^{-1} \circ F_{\lambda}$, and set $\xi^{\kappa,\Rc}_\cdot = \kappa_{\#}\xi^\Rc_\cdot$. It follows that $\xi^{\kappa,\Rc}_0 = \mu$, and $\xi^{\kappa,\Rc}$ is a terminating MVM. As above, we will call this time-scale the \emph{canonical-time} Bass-Root embedding (associated to the barrier $\Rc$), and note that there is a corresponding \emph{natural-time} Bass-Root embedding. As above, we will typically work with the canonical-time Bass-Root MVM.

  Of course, both the Bass and Root MVMs are special cases of the Bass-Root MVM (corresponding, respectively, to the cases where $\Rc = \{(t,x): t \ge 1\}$ and $h(x)=x$, or, equivalently, $\lambda=\eta_0$ and $\lambda=\mu$).
\end{itemize}

\section{MVMs in Wasserstein space and the first optimality property}

We will now study the properties of MVMs in Wasserstein spaces. To this end we define the set of \emph{transport plans} from $\R^d$ to $\R^d$ which couple measures $\lambda, \mu\in\Pc_1$:
\begin{equation}\label{eq:def_gamma}
  \Pi(\lambda,\mu) := \big\{ \nu \in \Pc(\R^d \times \R^d) : \nu(A\times \R^d) = \lambda(A),~ \nu(\R^d\times A) = \mu(A)\big\}.
\end{equation}
For any $p\ge 1$, the $p$-Wasserstein metric (see (7.1.1) in \cite{Ambrosio:2008aa}) on $\Pc_p$ is then given by:
\begin{equation}\label{eq:wasserstein_metric}
  \Wc_p^p(\lambda,\mu) := \inf\Big\{\int |x_1-x_2|^p\, \di \nu(x_1, x_2): \nu\in\Pi(\lambda,\mu)\Big\}.
\end{equation}
We note that the infimum is attained, in particular, the set $\Pi_o^p(\lambda,\mu)$ of optimisers is non-empty, closed and compact in the weak topology.
	
In this section we study the evolution over time of the Wasserstein distance between two MVMs; specifically, for any two MVMs $\xi,\eta\in\Pc_p$, $p\ge 1$, we are interested in the stochastic process $t\mapsto \Wc_p(\xi_t,\eta_t)$.  It is clearly adapted. Moreover, recall from Section \ref{sec:basic} that we may -- and do -- consider versions of $\eta$ and $\xi$ which are \cadlag{} in the topology induced by $\Wc_p$; an application of the triangle inequality then yields that the process $\Wc_p(\eta_\cdot,\xi_\cdot)$ is right-continuous.  We further adopt the following convention: for any two \emph{random} measures $\lambda$ and $\mu$ on $\R^d$, we say that a (random) measure on $\R^d\times\R^d$ is a transport plan from $\lambda$ to $\mu$, and write $\nu\in\Pi(\lambda,\mu)$, if the assertions in \eqref{eq:def_gamma} hold a.s.  For any $t\ge 0$, we then have
\begin{align}\label{eq:wasserstein_process}
  \Wc_p^p(\xi_t,\eta_t) = \essinf \Big\{\int |x_1-x_2|^p\, \di \nu(x_1, x_2): \nu\in\Pi(\xi_t,\eta_t)\Big\}\quad \mathrm{a.s.}		
\end{align}

The first result we have is the following:

\begin{proposition}\label{lemma:submtg_W}
  Let $p\ge 1$ and let $\xi^1_t,\xi^2_t\in\Pc_p(\R^d)$, $t\ge 0$, be two measure-valued martingales. Then, 
  \begin{itemize}
  \item[i)] the process $\Wc_p^p(\xi^1_t,\xi^2_t)$, $t\ge 0$, is a sub-martingale;
  \item[ii)] if $\xi^1$ and $\xi^2$ satisfy, for some family of measures $m(x,\dy)$, $x\in\R^d$,
    \begin{equation}	\label{eq:def_disintegration_optimum}
      \xi^2_t(\dy) = \int \xi^1_t(\di x)m(x,\di y), \quad t\ge 0,
    \end{equation}
    then $\Wc_p^p(\xi^1_t,\xi^2_t)$, $t\ge 0$, is a martingale;
  \item[iii)] for any MVM $\xi^1_t\in\Pc_p(\R^d)$, $t\ge 0$, and measure $\mu\in\Pc_p(\R^d)$, there exists an MVM $\xi^2_t\in\Pc_p(\R^d)$, $t\ge 0$, of the form \eqref{eq:def_disintegration_optimum} with $\xi^2_0=\mu$.
  \end{itemize}
\end{proposition}

\begin{proof}
  i)~~ Let $0<s<t$ and note that for any $\nu(\dx_1,\dx_2) \in \Pi(\xi^1_t,\xi^2_t)$, by use of Fubini's theorem, we may define an $\Fc_s$-measurable random measure $\nu_s$ by setting for bounded and measurable $f:\R^d\times\R^d\to\R$
  \begin{equation}\label{eq:def_nu_s_fubini}
    \int f(x_1, x_2) \, \nu_s(\dx_1,\dx_2) = \Ep{\int f(x_1, x_2)\, \nu(\dx_1,\dx_2)|\Fc_s}.
  \end{equation}
  Then, in particular, 
  \begin{align*}
    \nu_s(A\times\R^d)=\E[\nu(A\times\R^d)|\F_s]=\E[\xi^1_t(A)|\F_s]=\xi^1_s(A), \;\;\mathrm{a.s.},\;\; A\in\Bc(\R^d),
  \end{align*} 
  and it follows that $\nu_s$ is an element of $\Pi(\xi^1_s,\xi^2_s)$.  Next, since the set $\{\int |x_1-x_2|^p\, \di \nu(x_1, x_2): \nu\in\Pi(\xi^1_t,\xi^2_t)\}$ is directed downwards, there exists a sequence $\nu^n\in\Pi(\xi^1_t,\xi^2_t)$ such that $\int |x_1-x_2|^p\, \di \nu^n(x_1, x_2)\searrow\W_p^p(\xi^1_t,\xi^2_t)$ a.s. By use of the monotone convergence theorem, we thus obtain
  \begin{align*}
    \Ep{\W_p^p(\xi^1_t,\xi^2_t)|\Fc_s}  
    & = \Ep{\lim_{n\to\infty} \int |x_1- x_2|^p\, \nu^n(\dx_1,\dx_2)|\Fc_s}\\
    & = \lim_{n\to\infty} \int |x_1- x_2|^p\, \nu^n_s(\dx_1,\dx_2)
      ~\ge~ \Wc_p^p(\xi^1_s,\xi^2_s),
  \end{align*}
  where $\nu^n_s$ is defined via \eqref{eq:def_nu_s_fubini} with respect to $\nu^n$.
  
  ii) Let $t>0$ and define $\nu(\di x,\di y):=\xi^1_t(\di x)m(x,\di y)$; by \eqref{eq:def_disintegration_optimum} we have that $\nu\in\Pi(\xi^1_t,\xi^2_t)$. By use of i) we then obtain
  \begin{align*}
    \Wc^p_p\left(\xi^1_0,\mu\right)
    & = \int |x-y|^p\xi^1_0(\di x)m(x,\di y) \\
    &= \E\left[\int |x-y|^p\xi^1_t(\di x)m(x,\di y)\right]
      ~=~
      \E\left[\Wc^p_p\left(\xi^1_t,\xi^2_t\right)\right].
  \end{align*}
  
  iii) Recall that the set of minimizers for \eqref{eq:wasserstein_metric} is non-empty and let $\nu_0\in\Pi_o^p(\xi^1_0,\mu)$. It follows by disintegration (see e.g. \citep[Theorem~5.3.1]{Ambrosio:2008aa}) that there exists a family of measures $m^0(x,\di y)$, $x\in\R$, such that $\nu_0(\di x,\di y)=\xi^1_0(\di x) m^0(x,\di y)$.  Then, define
  \begin{align*}
    \xi^2_t:=\int \xi^1_t(\di x)m^0(x,\di y),\quad t\ge 0.
  \end{align*}
  Since
  \begin{equation*}
    \int \xi^1_s(\di x)m(x,\di y)
    ~=~ \E\left[\int \xi^1_t(\di x)m(x,\di y)|\F_s\right],\quad s\le t,
  \end{equation*}
  the thus defined process $\xi^2\in\Pc_p$ is indeed a measure-valued martingale; by definition, it satisfies \eqref{eq:def_disintegration_optimum} and the condition $\xi^2_0=\mu$.
\end{proof}

We note that the above result does not require the measure-valued martingales to be terminating; in particular, it holds also for $\xi^2_t=\lambda\in\Pc_1$ constant.

\subsection{The first optimisation problem: optimality of the Bass embedding}

We first recall the following notation from Section \ref{sec:ex}: $B_t$ denotes a $1$-dimensional Brownian motion, $\eta_t:=\mathrm{Law}(B_1|\F_t)$ and for any $\mu\in\Pc_1$ and $h:\R\to\R$ such that $h_\#\eta_0=\mu$, $\xi^h_t:=\mathrm{Law}(h(B_1)|\F_t)$ defines a finitely terminating measure-valued martingale with $\xi^h_0=\mu$ -- the canonical-time Bass MVM. In particular,
\begin{equation} \label{eq:eta_h}
  \xi^h_t=h_{\#} \eta_t.
\end{equation} 
		
As a consequence of Proposition \ref{lemma:submtg_W}, we then obtain the following characterisation of the Bass embedding as the solution to a particular optimisation problem:

\begin{theorem}\label{prop:opt_1}
  Given $p\ge 1$, let $\mu\in\Pc_p(\R)$ be a given atomless measure. Define $h(x):=F_\mu^{-1}\circ F_{\eta_0}(x)$, $x\in\R$, and let $\xi^h_\cdot$ be the associated canonical-time Bass MVM.  Then, $\xi^h_\cdot$ minimises, simultaneously for all weight-functions $w:[0,1]\to\R$,
  \begin{equation}\label{eq:bass_criterion_1}
    \E\left[\int_0^1 w(t) \Wc_p^p(\eta_t,\xi_t)\di t\right],
  \end{equation}
  over all measure-valued martingales $\xi_t$ with $\xi_0=\mu$.
\end{theorem}

\begin{proof}
  According to part ii) of Proposition \ref{lemma:submtg_W}, $\Wc_p^p(\eta_t,\xi^h_t)$ is a martingale; specifically, for $t\ge 0$, $(\text{Id}\times h)_{\#}\eta_t\in\Pi(\eta_t,\xi^h_t)$ and
  \begin{equation*}
    \Wc^p_p\left(\eta_0,\mu\right)
    =
    \int |x-h(x)|^p\eta_0(\di x)
    =
    \E\left[\int |x-h(x)|^p\eta_t(\di x)\right]
    =
    \E\left[\Wc^p_p\left(\eta_t,\xi^h_t\right)\right].
  \end{equation*}
  On the other hand, according to Proposition \ref{lemma:submtg_W} i), $\Wc_p^p(\eta_t,\xi_t)$ is a sub-martingale for any other measure-valued martingale $\xi$. Hence, given that $\xi_0=\mu$, $\E[\Wc^p_p(\eta_t,\xi_t)]\ge \E[\Wc^p_p(\eta_t,\xi^h_t)]$, for $t\ge 0$.  Integration with respect to $w(t)$ and use of Fubini yields the result.
\end{proof}

\begin{remark}
  We note that for $w(t)=\delta_1$, the criterion in \eqref{eq:bass_criterion_1} reduces to minimising $\E[(B_1-\M(\xi_1))^p]$ over measure-valued martingales with $\xi_0=\mu$ and terminating before $t=1$.  In particular, for the case $p=2$, since the law of $B_1$ and $\M(\xi_1)$ are fixed, this is equivalent to maximising $\E[B_1\M(\xi_1)]$.
\end{remark}

\section{Markov properties of MVMs}

In this section, we consider certain natural `Markov-like' properties of MVMs. Clearly, one can simply ask that the MVM is itself a Markov process, in the usual sense. This is closely related to the definition of the Markov property due to \citet{Eldan:2016aa}, who gave a definition of the (time-homogenous) Markov property for the (closely related) notion of a `Skorokhod embedding scheme'. An example of an embedding scheme/MVM that has this Markovianity is the Root embedding/MVM, where the process is completely determined by the current value of the MVM, through the corresponding Root barrier. As the MVM evolves, the barrier will move to reflect the current measure, but will evolve consistently. Note that this evolution can be defined from any given starting measure. (We note that \cite{Eldan:2016aa} also required a shift invariance, so that the evolution of the process was invariant to constant shifts).

In fact, we will need a different form of the Markov property: specifically, we want to be able to control the evolution of the MVM in terms of the mean of the process, in such a way that if the mean of the process is at a given level at a given time, then we can conclude the value of the process at that time; however we then require a property relating the motion of the process to the motion of its mean which will imply that the mean process is (time-inhomogeneous) Markovian. Note that in this formulation, it will be important to have a fixed starting point for the process.

The required property is:
\begin{definition}[Lipschitz-Markov]\label{def:ML}
  We say that a MVM is Lipschitz-Markov if
  \begin{equation}	\label{eq:def_ML}
    \W_1(\xi_t(\omega),\xi_t(\omega')) = |\M(\xi_t(\omega))-\M(\xi_t(\omega'))|.	
  \end{equation}
\end{definition}

We note that \eqref{eq:def_ML} holds as an inequality for any MVM $\xi$. Indeed, it follows from the duality of the Wasserstein distance (Theorem 6.1.1, \cite{Ambrosio:2008aa}) that for any two measures $\lambda,\mu\in\Pc_1$,
\begin{align} \label{eq:ineq_W1_mean}
  \W_1(\lambda,\mu)\; 
  & =  \sup_{f \text{ 1-Lipschitz}}\left\{\int f(x) \, \di (\lambda - \mu)\right\} \nn\\
  & \ge  \max\left\{\int x \, \di (\lambda - \mu), \int (-x) \, \di (\lambda - \mu)\right\}
    ~ = ~ |\M(\lambda) - \M(\mu)|.
\end{align}

The Lipschitz-Markov property has appeared in e.g. \cite{hirsch2012kellerer,lowther2009limits}, in the following form: an adapted process, say $M$, is called Lipschitz-Markov if, for any bounded $1$-Lipschitz function $f:\R\to\R$, there exists a $1$-Lipschitz function $g:\R\to\R$, such that
\begin{equation}
  g(M_s)=\E[f(M_t)|\F_s], \qquad s\le t. 
\end{equation}
In the form of Definition \ref{def:ML}, the Lipschitz-Markov property first appeared in \cite{beiglbock2015root}, where, in particular, its relation to the Root embedding was studied.  It follows from the dual representation of the $\W_1$-distance, that a MVM $\xi_t$ is Lipschitz-Markov if and only if $\M(\xi_t)$ is Lipschitz-Markov in the sense of \cite{hirsch2012kellerer,lowther2009limits}; see \citep[Lemma~5]{beiglbock2015root}.

In particular, every Lipschitz-Markov MVM has a Markov mean process (see \cite[Remark~1.70]{Liggett:2010aa}).  
We also have the following result:
	
\begin{lemma} \label{lem:mark-lip} A MVM is Lipschitz-Markov if and only if it is of the form $\xi_t(\omega) = m_t(\M(\xi_t(\omega)))$ for some function $m_t$ and there is an isotonic map $T$ (\ie{} $x \le T(x)$), which may depend on $t, x$ and $y$, such that $m_t(x) = T_{\#}m_t(y)$ whenever $x>y$.
	
  In particular, the Root, Bass and Bass-Root MVMs are Lipschitz-Markov.
\end{lemma}

\begin{proof}
  To argue the sufficiency, suppose w.l.o.g. that $\M(\xi_t(\om'))\ge \M(\xi_t(\om))$. Then, the existence of an isotonic map such that $\xi_t(\om')=T_\#\xi_t(\om)$, implies that
  \begin{align*}
    \W_1(\xi_t(\om'),\xi_t(\om))
    &\le\int|T(x)-x|\di\xi_t(\om)\\
    &=\int T(x)\di\xi_t(\om)-\int x\di\xi_t(\om)
      ~=~ |\M(\xi_t(\om'))-\M(\xi_t(\om))|; 
  \end{align*}
  equalities must thus hold throughout according to \eqref{eq:ineq_W1_mean}. 
  For the necessity, we refer to the discussion immediately preceding Lemma~4 of \cite{beiglbock2015root}.

  That the Root MVM is Lipschitz-Markov then follows as in the proof of this same result.

  To prove that the Bass MVM is Lipschitz-Markov, recall that $\xi^h_t=h_{\#}\eta_t=h_\#\mathcal{N}(\M(\eta_t),1-t)$; if $h$ is monotone, then so is $\M(\xi^h_t)$ viewed as a function of $\M(\eta_t)$. Therefore $\M(\eta_t)$, and thus also $h_{\#}\eta_t$, may be recovered from $\M(\xi^h_t)$. Moreover, the existence of the isotone map $T$, then follows from the monotonicity of this mapping combined with the monotonicity of $h$ itself. The Bass-Root case follows by combining the two arguments.
\end{proof}

\begin{remark}
  We note that we have several notions of Markovianity here. To try and clarify the situation, we discuss this property in relation to the MVMs introduced earlier. We consider first the Markov property of the MVM as a process in the space of probability measures (analogous to the case considered by \cite{Eldan:2016aa}). In this case, it is easy to see that the Root embedding is a time-homogenous Markov process (this is essentially shown in \cite{Eldan:2016aa}). If we consider the Bass MVM in the form $\xi^h_t = h_\# \eta_t$, where $\eta_t = \Nc(B_t,1-t)$, then it is easy to see that $\xi^h_t$ is a time-inhomogeneous Markov process. On the other hand, in the time-changed version where $\tilde\xi^h_t:= \xi^h_{\tau_t}$ is chosen such that $\langle \M(\tilde\xi^h_\cdot)\rangle_t = t$, then again, arguments of Eldan show that $\tilde\xi^h_t$ is a \emph{time-homogenous} Markov process.

  However, if we fix the starting law $\xi_0$, and hence $h$, and consider just the mean process $\M(\xi^h_t)$, then it is possible to check that this process is a time-inhomogeneous Markov process. On the other hand, the process $\M(\tilde{\xi}^h_t)$ is not a Markov process. Similar arguments show that the mean process corresponding to the Bass and Bass-Root (on the underlying Brownian time-scale) are also both time-inhomogenous Markov processes. In the alternative Bass-Root case, where the time-scale is fixed so that the mean is a Brownian motion up to stopping, the resulting process is not Markov (in particular, the stopping time is not generally adapted to the filtration of the mean process).
\end{remark}

\section{The notion of `Speed' of MVMs and the second optimality property}

\subsection{The speed of an MVM}

A key question is how to define a suitable notion of `speed' for MVMs. That is, a means by which one can measure how fast a MVM evolves over time. We want a notion of speed which replaces the usual notion of quadratic variation, but which is strictly increasing whenever the MVM evolves. As we will see below, this is not the case for the quadratic variation of the mean process $\M(\xi_t)$. However, we will be able to define a notion of speed such that, in certain circumstances,  the speed of our MVMs and the quadratic variation of the mean process coincide, while in general the speed will dominate the quadratic variation.  These features will be important for our second optimality property.

We first make some definitions relating to Lipschitz functions: we write $C^{0,1}(\R^d)$ for the set of Lipschitz functions $f:\R^d \to \R$ equipped with the norm:
\begin{equation*}
  ||f||  = \sup_{x \in \R^d} |f(x)| + \sup_{x \neq y} \frac{|f(x)-f(y)|}{|x-y|}.
\end{equation*}
In addition, we write $\Lip^1$ for the (closed) subset $\left\{f \in C^{0,1}| \sup_{x \neq y} \frac{|f(x)-f(y)|}{|x-y|} \le 1\right\}$. 

We now introduce the notion of a simple Lipschitz-valued process: we say that the process $(f_t)_{t \ge 0}$ is in the set $\SL^1$ if
\begin{equation*}
  f_t(x) = \sum_{n=0}^\infty f_i(x) \indic{t \in (\tau_i, \tau_{i+1}]}
\end{equation*}
where $f_i \in \Lip^1$ is $\Fc_{\tau_i}$-measurable, and $(\tau_i)_i$ is a sequence of increasing stopping times such that $\P(\lim_{n \to \infty} \tau_n = \infty) = 1$.

To a process $(f_t)_t \in \SL^1$ we associate a martingale (effectively, a stochastic integral against $\xi_t$)\footnote{There is a substantial literature on the construction of stochastic integrals for the closely related case of martingale measures (effectively, when $\xi_t$ can be a signed measure, instead of a probability measure), however this literature is dominated by the cases where either $\xi_t$ satisfies an orthogonality condition, or  possesses \emph{nuclear covariance}; in our setup, neither condition is natural. We refer to \cite{Walsh:1986aa} for details. A natural question is whether the stochastic integral defined for the class $\SL^1$ of simple processes can be extended to a natural limit.} defined by:
\begin{equation*}
  I_t := \sum_{i} \left( \int f_{i}(x) \, \xi_{\tau_{i+1}\wedge t}(\dx) - \int f_{i}(x) \, \xi_{\tau_i\wedge t}(\dx)\right).
\end{equation*}
Since $\xi_t \in \Pc_1$ for each $t$, and $f_i$ are 1-Lipschitz, this sum is well defined and the resulting process is a martingale. Then we can associate an increasing process $A^f_t$ to each $f \in \SL^1$ by defining $A^f$ to be the compensator of the supermartingale $I_t^2$.

\begin{definition}
  We define the \emph{speed} of the MVM $\xi_t$ to be the smallest increasing, \cadlag{}, adapted process $[\xi]_t$ such that $[\xi]_t \ge A^f_t$ almost surely, for every $f \in \SL^1$.
\end{definition}

Note that such a process is trivially well defined by the fact that the set of increasing processes is directed downwards, so if $A_t, \tilde{A}_t$ are both increasing candidates for the speed, then so too is $A_t \wedge \tilde{A}_t$. In particular, it is easy to see that $\di [\xi]_t \ge \di A_t^f$, for any $f \in \SL^1$.

\begin{proposition}\label{prop:speed}
  Suppose $\xi$ is an MVM with speed $[\xi]$. Then
  \begin{enumerate}
  \item\label{item:2} $[\xi]_t \ge \langle \xi \rangle_t := \langle \M(\xi)\rangle_t$, the quadratic variation of the mean process;
  \item\label{item:1} $[\xi]_t \equiv 0$ if and only if $\xi$ is constant;
  \item if $(\tau_t)_{t \ge 0}$ is a time-change, and $\tilde{\xi}_t := \xi_{\tau_t}$, then $[\tilde{\xi}]_t = [\xi]_{\tau_t}$, that is, the speed is invariant under time-change.
  \end{enumerate}
\end{proposition}

We note that the condition in (\ref{item:1}) is not true for the quadratic variation process, as the subsequent example demonstrates:
\begin{example}\label{ex:speed_not_equal_qv}
  Let $\xi_t$ be an MVM in a Brownian filtration with $\xi_0 = U([-1,1])$. Between time 0 and the first hitting time of $\pm \eps$, we use the Brownian motion to flip an artificial coin, based on which we have either $\xi = U([-1,-1/\sqrt{2}]\cup [0,1/\sqrt{2}])$ or the converse. Then for this example, the mean process does not move at all, but the distance changes a lot. Variants on this can be used to get small movement in the quadratic variation, but large movements in speed in various natural ways. In particular, we see that it is certainly not always true that $[\xi]_t= \langle \xi \rangle_t$.
\end{example}

\begin{proof}[Proof of Proposition~\ref{prop:speed}]
  The first part of the result follows immediately from considering a simple process such that $f_t(x) = x$ for all $t$.

  To see (\ref{item:1}), note that if $\xi$ is not constant almost surely, then there exist $h \in \Lip^1$ and $t_1 < t_2$ such that $\int h(x) \, \xi_{t_1}(\dx) \neq \int h(x) \, \xi_{t_2}(\dx)$ with positive probability. Taking $f_t(x) = h(x) \indic{t \in (t_1, t_{2}]}$, then $f \in \SL^1$ gives a non-trivial martingale, and hence a non-zero compensator.

  Finally, the last part follows immediately from the definition of the speed of the process.  
\end{proof}

\subsection{The class of consistent MVMs}\label{sec:consistent}
	
As demonstrated by Example \ref{ex:speed_not_equal_qv}, the speed of an MVM will typically not coincide with the quadratic variation of its mean process. In this section we study a class of MVMs for which this is however the case. Specifically, the following result yields a sufficient condition for the speed and quadratic variation of a process to actually coincide:
	
\begin{definition}\label{def:consistent}
  We call a Lipschitz-Markov MVM $\xi_t$ consistent if its associated function $m_t$ satisfies $\W_1(m_s(x),m_t(x)) \le \alpha(\eps)$ for all $x$ and $|t-s|<\eps$, for some function $\alpha:\R_+\to\R_+$ such that $\alpha(\eps) \eps^{-1/2} \to 0$ as $\eps \to 0$.
\end{definition}

\begin{lemma}\label{lemma:consistent_LM}
  If $\xi_t$ is a consistent Lipschitz-Markov MVM, then $[\xi]_t = \langle \xi\rangle_t$.
\end{lemma}

\begin{proof}
  Let $g \in \Lip^1$. By the Lipschitz-Markov property of $\xi_{\cdot}$, for $s < t$ we get
  \begin{align*}
    \left|\int g(x)\left(\xi_{t}-\xi_{s}\right)(\dx)\right|
    & \le \W_1 (\xi_{t},\xi_{s})\\
    & \le \W_1 \left(\xi_{t},m_{t}(\M(\xi_{s}))\right) + \W_1 \left(m_{t}(\M(\xi_{s})),\xi_{s}\right)\\
    & \le |\M(\xi_{t})-\M(\xi_{s})| + \alpha(t-s).
  \end{align*}

  Now take $f \in \SL^1$. Possibly taking a finer (stochastic) partition $\tau_1 \le \tau_2 \le \dots$, which includes the original partition associated with $f$, and also the time $t$, and using the fact that $\alpha(\eps)^2 < \delta \eps$, for some $\delta >0$\footnote{For later use, we note that this identity holds even if $\delta$ is a fixed constant.}, we have: \reqnomode
  \begin{align}  
    & \hspace{-1cm}\sum_{i}\left( \int f_i(x)\left(\xi_{\tau_{i+1}\wedge t}-\xi_{\tau_{i}\wedge t}\right)(\dx)\right)^2  \label{eq:proof_consistency_2} \\
      & \le  \sum_{i:\tau_i<t} \big[\alpha(\tau_{i+1} -\tau_i)^2 + 2 \alpha(\tau_{i+1} -\tau_i) |\M(\xi_{\tau_{i+1}})-\M(\xi_{\tau_i})| \nonumber \\
    & \qquad \qquad {} + |\M(\xi_{\tau_{i+1}})-\M(\xi_{\tau_i})| ^2   \big] \nonumber
    \\
    & \le  \delta t  + 2\bigg(\sum_{i:\tau_i<t} \alpha(\tau_{i+1} -\tau_i)^2\bigg)^{\frac{1}{2}} 
      \bigg(\sum_{i:\tau_i<t} |\M(\xi_{\tau_{i+1}})-\M(\xi_{\tau_i})|^2\bigg)^{\frac{1}{2}} \nonumber\\ 
    & \qquad \qquad {} + \sum_i |\M(\xi_{\tau_{i+1}\wedge t})-\M(\xi_{\tau_i\wedge t})| ^2 \nonumber\\
    & \le  \delta t + 2\sqrt{\delta t  \sum_i |\M(\xi_{\tau_{i+1}\wedge t})-\M(\xi_{\tau_i\wedge t})|^2} + \sum_i |\M(\xi_{\tau_{i+1}\wedge t})-\M(\xi_{\tau_i\wedge t})| ^2, \nonumber
  \end{align}
  \leqnomode
  where we used Cauchy-Schwarz. Note that for $\delta>0$ arbitrarily small, \eqref{eq:proof_consistency_2} still holds for $\varepsilon$ sufficiently small. 
  Further, we may find a sequence of partitions, with $|\Pi_n|\to 0$, such that $A^f_t = \lim_{n\uparrow\infty} \sum|I_{\tau^n_{i+1}\wedge t}-I_{\tau^n_{i}\wedge t}|^2$ and $\langle\xi\rangle_t = \lim_{n\uparrow\infty} \sum|\M(\xi_{\tau^n_{i+1}\wedge t})-\M(\xi_{\tau^n_{i}\wedge t})|^2$, \as{}. 
  Hence, we deduce that $A^f_t \le \langle \xi\rangle_t$, $t\ge 0$.  
  Since $f\in\SL^1$ was arbitrary, we thus have that $\langle \xi\rangle_t$ is an increasing process dominating $A^f_t$ for all $f\in\SL^1$. It follows that $[\xi]_t \le \langle \xi\rangle_t$, and applying Proposition \ref{prop:speed} we conclude.
\end{proof}

\begin{corollary} \label{cor:cont}
  A consistent Lipschitz-Markov MVM $\xi$ is continuous if its mean process, $\M(\xi)$, is continuous.
\end{corollary}

\begin{proof}
  Let $f \in \SL^1$ correspond to the constant process, $f_i = g \in \Lip^1$ for all $i$. It is sufficient to show that $A^f$ has no jumps, but by Lemma~\ref{lemma:consistent_LM}, $\di A^f_t \le \di \langle \xi \rangle_t$, and the conclusion follows.
\end{proof}

\begin{proposition} \label{prop:Consistent} The Bass-Root MVMs with Lipschitz scale function $\kappa$ (and hence also the Bass and Root MVMs) are continuous Lipschitz-Markov MVMs with
  \begin{align*}
    \langle \xi^{\kappa,\Rc}\rangle_t = [\xi^{\kappa,\Rc}]_t.
  \end{align*} 
\end{proposition}

\begin{proof}
  Consider a Bass-Root MVM, with $K$-Lipschitz scale function $\kappa$, that is $\xi^{\kappa,\Rc}_t = \kappa_{\#}\xi^\Rc_t$ for an appropriate Root MVM $\xi^\Rc_t$ with barrier $\mathcal{R}$; recall that $\tau_\Rc=\inf\{t\ge 0:\xi^\Rc_t\in\Pc^s\}$.  Let $t>0$ and $x\in\R$ be such that $d((t,x),\mathcal{R})>0$, where $d$ denotes the Euclidean distance, and let $0<s<t$.  In turn, let $B^{s,x}$ and $B^{t,x}$ be two independent BMs starting respectively in $(s,x)$ and $(t,x)$.  Define $m_s(x):=\Lc(\kappa(B^{s,x}_{\tau^s}))$, with $\tau^s$ the first time $B^{s,x}$ hits the barrier $\mathcal{R}$, and let $m_t(x)$ be analogously defined.
	
  We first derive an appropriate bound on $\Wc_1(m_s(x),m_t(x))$.  To this end, let $\delta>0$ and $a>0$ such that the box $[s,t+\delta]\times[x-a,x+a]$ is disjoint from $\mathcal{R}$. Note that we have the following finite bound
  \begin{align}\label{eq:kappa}
    \beta_0:=\sup\left\{\E^{r,y}\left[\left|\kappa(B_{\tau^{r}}^{r,y})\right|\right]:(r,y )\in [s,t+\delta]\times[x-a,x+a]\right\}<\infty.
  \end{align}
  In order to compare the laws of $\kappa(B^{s,x}_{\tau^s})$ and $\kappa(B^{t,x}_{\tau^t})$, we couple the paths originating from $(s,x)$ and $(t,x)$, respectively, and exiting the box $[s,t+\delta]\times[x-a,x+a]$ at the same points as those will necessarily contribute to the same law. It then suffices to find an appropriate bound on the mass that we may not couple in this way.  To this end, let $\beta>0$, and note that according to Lemma \ref{lem:appendix_estimates} (i), the probability of $B^{s,x}$ exiting the box before $t$ is bounded by $2 \beta (t-s)$, provided that $t-s<\Delta$, for some sufficiently small $\Delta>0$.  Further, applying Lemma \ref{lem:appendix_estimates} (iii), choosing if necessary $\delta$ and $\Delta$ even smaller, the total mass corresponding to paths exiting the box at the `top' and `bottom' but not being coupled is bounded by $\beta(t-s)$ for $t-s<\Delta$.  Finally, applying Lemma \ref{lem:appendix_tv}, we may bound the total mass corresponding to paths exiting the box along the line $(x,t+\delta)$, $x\in(-a,a)$, and not being coupled, by $\sqrt{2/\pi}\frac{t-s}{\delta}$; indeed, although the presence of the `top' and `bottom' of the box naturally modifies the distribution of $B^{s,x}$ and $B^{t,x}$ along this line, the corresponding contribution to the total variation distance is compensated by the fact that the `tails' for $x\not\in(-a,a)$ are cut off.  Summing up and applying \eqref{eq:kappa}, we obtain that there exists $\delta>0$ and $\Delta>0$, such that for $t-s<\Delta$,
  \begin{align*}
    \Wc_1\left(\Lc\left(\kappa(B^{s,x}_{\tau^s})\right),\Lc\left(\kappa(B^{t,x}_{\tau^t})\right)\right) 
      & \le \beta_0 \left(3\beta + \sqrt{2/\pi}\delta^{-1}\right) (t-s).
  \end{align*}
  In consequence, for any $d>0$, there exist $\Delta^d>0$ and $\beta^d>0$ such that for any $0<s<t$ and $x>0$ with $d((t,x),\mathcal{R}))>d$ and $t-s<\Delta^d$, 
  \begin{align*}
    \Wc_1\left(m_s(x),m_t(x)\right)\le \beta^d (t-s).
  \end{align*} 
  Applying Lemma \ref{lemma:consistent_LM}, we obtain $[\xi^{\kappa,\Rc}]_t=\langle \xi^{\kappa,\Rc}\rangle_t$, for $t\le \tau^d:=\inf\{t\ge 0: d((\mathbb{M}(\xi^\Rc_t),t),\mathcal{R})\le d\}$. Since $d>0$ was arbitrarily chosen, it follows that $[\xi^{\kappa,\Rc}]_t=\langle \xi^{\kappa,\Rc}\rangle_t$, for $t<\tau_\Rc$.
  
  It only remains to argue that the speed $[\xi^{\kappa,\Rc}]_\cdot$ does not possess a jump at $t=\tau_\Rc$.  To this end, consider the Root barrier corresponding to a single atom. Let $s<t$ and consider now the same BM starting in $(s,x)$ and $(t,x)$ and couple the corresponding paths. Then, using that $\kappa$ is $K$-Lipschitz we have
  \begin{align*} 
    \Wc_1(m_s(x),m_t(x))\le K \E [|B^{s,x}_{\tau^s}-B^{t,x}_{\tau^t}|] \le K\E [L^0_{t-s}] \le K\sqrt{t-s};
  \end{align*} 
  we note that this bound is improved for any other Root barrier.  In particular, the first inequality in \eqref{eq:proof_consistency_2} holds for $\alpha(t) = t$, and it follows that $\di [\xi^{\kappa,\Rc}]_t \le \tilde \beta (\dt+\di\langle\xi^{\kappa,\Rc}\rangle_t)$, for some $\tilde \beta>0$, and since $\langle\xi^{\kappa,\Rc}\rangle_t$ is continuous for the Bass-Root embedding we may conclude.
\end{proof}

\begin{remark}
  Yet another natural candidate for measuring the `speed' of a MVM, is given by the following increasing process:
  \begin{align}\label{eq:speed_alt}
    \widetilde{[\xi]}_t = \liminf_{|\Pi_n| \to 0} \sum_{i: t_{i+1}<t} \W_1^2(\xi_{t_{i+1}},\xi_{t_i}),
  \end{align}
  where the limit is taken over a fixed sequence $\Pi_n = \{0 = t_0 \le t_1 \le \dots , t_k \to \infty \text{ as } k \to \infty\}$ of partitions of $\R_+$ with $|\Pi_n| = \sup_{i} |t_{i+1}-t_i|$.
  Recall that by the duality of the Wassertein distance, 
  $|\Wc_1(\xi_{t_{i+1}},\xi_{t_i})| = \sup\{|\xi_{t_{i+1}}(f)-\xi_{t_i}(f)|:f\in \textrm{$\mathcal{L}ip^1(\Fc_{t_{i+1}})$}\}$,
  which yields $\widetilde{[\xi]}_t\ge [\xi]_t$. For consistent MVMs, we immediately see that equality holds. More generally, regardless of whether the speed coincides with the quadratic variation or not, we expect equality to hold for a large class of MVMs being adapted to the Brownian filtration.
\end{remark}

\subsection{The $\rho$-speed of an MVM}

We observe that, although important in some of our definitions (in particular, through the use of the Wasserstein topology), the metric of the underlying space ($\R^d$) so far has not been too significant. In this section, we consider one way of transforming the underlying geometry of the space, and examine the consequences for the processes. The starting point for this observation is that if $(\xi_t)_t$ is an MVM, and $\rho$ is an increasing function such that $|\rho(x)| \le K(1+|x|)$ for some $K>0$, then $\chi_t := \rho_\# \xi_t$ is also a MVM. This follows immediately from the definition of the pushforward, so that $\int f(x) \, \chi_t(\dx) = \int f\circ\rho(x) \, \xi_t(\dx)$.

Of course, to the process $\chi$, we can associate a speed process, $[\chi]_t$, and we can think of this process as the $\rho$-speed of the process $\xi$. This corresponds to the following definition:

We say that $f \in \Lip^{1,\rho}$ if and only if $f = g\circ\rho$,
where $g \in \Lip^1$. Then $(f_t) \in \SL^{1,\rho}$ if $f$ is of the form
\begin{equation*}
  f_t(x) = \sum_{n=0}^\infty f_i(x) \indic{t \in (\tau_i, \tau_{i+1}]}
\end{equation*}
where $f_i \in \Lip^{1,\rho}$ is $\Fc_{\tau_i}$-measurable, and $(\tau_i)$ is a sequence of increasing stopping times such that $\P(\lim_{n \to \infty} \tau_n = \infty) = 1$. Then we have:
\begin{definition}
  The \emph{$\rho$-speed} of the MVM $\xi$ is the smallest increasing, \cadlag{}, adapted process $[\xi]^\rho_t$ such that $[\xi]_t^\rho \ge A_t^f$ almost surely, for every $f \in \SL^{1,\rho}$.
\end{definition}
It is a simple consequence of the definitions and the argument  that $[\xi]^\rho_t \ge \langle \M^\rho(\xi)\rangle_t$, where $\M^\rho(\xi) := \int \rho(x) \, \xi(\dx)$.

\subsection{The second optimisation problem: optimality of the Bass and modified Root solution}

As a consequence of the above properties of the ($\rho$)-speed, we obtain yet another optimality property of the Bass embedding. More generally, our next result provides an optimality criterion which is minimised by the (canonical-time) Bass-Root MVM: 

\begin{theorem}\label{thm:bass_2nd_opt}
  Let $\mu, \lambda\in\mathcal{P}^1(\R)$ be given atomless probability measures, and let $\xi^\Rc$ be the Root MVM corresponding to embedding $\lambda$. Define $\kappa = F_{\mu}^{-1}\circ F_{\lambda}$ and write $\rho = \kappa^{-1}$ for its right continuous inverse. Then the canonical-time Bass-Root MVM, $\xi^{\kappa,\Rc} = \kappa_{\#} \xi^\Rc$, is a terminating, continuous MVM with $\xi^{\kappa,\Rc}_0 = \mu$, which minimises
   \begin{align}\label{eq:bass_2nd_opt}
    \Ep{F\left(\left[\xi\right]_\infty^\rho\right)}
  \end{align}
for any increasing, convex function $F$, over the class of continuous, terminating MVMs with $\xi_0 = \mu$. \end{theorem}

In particular, we see that the (canonical-time) Bass MVM minimises \eqref{eq:bass_2nd_opt} for the choice of $\rho(x)=h^{-1}(x)$ with $h$ given as in Section \ref{sec:ex}, whereas the Root MVM minimises \eqref{eq:bass_2nd_opt} with the $\rho$-speed replaced by the non-modified speed $[\xi_\cdot]$.

\begin{proof}[Proof of Theorem \ref{thm:bass_2nd_opt}]
  The MVM $\xi^{\kappa,\Rc}$ is terminating with $\xi^{\kappa,\Rc}_0 = \mu$ by definition, and continuous by Corollary~\ref{cor:cont} and Proposition~\ref{prop:Consistent}. Moreover, $\rho_{\#} \xi^{\kappa,\Rc} = \xi^\Rc$, and so $[\xi^{\kappa,\Rc}]_{\infty}^\rho = [\xi^\Rc]_\infty$. But by Proposition~\ref{prop:Consistent}, we have $[\xi^\Rc]_\infty = \langle \xi^\Rc \rangle_\infty$, and so $\E[F([\xi^{\kappa,\Rc}]_\infty^\rho)] = \E[F(\langle\xi^\Rc\rangle_\infty)]$.  Note in particular, that $M_\cdot^\Rc := \M(\xi^\Rc_\cdot)$ is a stopped Brownian motion, stopped according to the first hitting time of a Root barrier, and also has stopped law $\lambda$.

  Now consider any other $\xi$ satisfying the conditions of the theorem. As above, we have $\rho_{\#} \xi_\cdot = \chi_\cdot$ for some terminating MVM $\chi$. Then $[\xi]^\rho_t = [\chi]_t \ge \langle\chi\rangle_t$, by Proposition~\ref{prop:speed} (\ref{item:2}). In particular, $M_\cdot := \M(\chi_\cdot)$ is a uniformly integrable, continuous martingale, with $M_\infty \sim \lambda$, and so $\Ep{F\left([\xi]^\rho_\infty\right)} \ge \E[F(\langle M\rangle_\infty)] \ge \E[F(\langle M^\Rc \rangle_\infty)]$, where the first step follows from the increase of $F$, and the second from convexity of $F$ and the optimality of the Root embedding, which completes the proof.
\end{proof}

\appendix
	
\section{Auxiliary estimates}

\begin{lemma}\label{lem:appendix_tv}
  Let $0<u<v$, then
  \begin{align*}
    ||\Nc(0,u)-\Nc(0,v)||_{\mathsf{TV}}\le \sqrt{\frac{2}{\pi}}\frac{v-u}{u}.
  \end{align*} 
\end{lemma}
	
\begin{proof}
  Denote the densities corresponding to $\Nc(0,u)$ and $\Nc(0,v)$ by $p_u$ and $p_v$, and denote by $\delta>0$ the value for which $p_u(\delta)=p_v(\delta)$. Then, the total variation distance between $\Nc(0,u)$ and $\Nc(0,v)$ is bounded by
  \begin{align}\label{eq:app_tv}
    ||p_u-p_v||_{\mathsf{TV}}
    & \le 4\int_{\delta}^\infty p_v(\dx)-p_u(\dx)\nonumber\\
    & = 4\left(p_v([\delta,\infty)) - p_v\left([\delta\sqrt{\frac{v}{u}},\infty)\right)\right)\nonumber\\
    & = 4 p_v\left([\delta,\delta\sqrt{\frac{v}{u}})\right) \;\le\; 4 p_v(0)\delta\left(\sqrt{\frac{v}{u}}-1\right).
  \end{align}
  We note that $\delta$ solves
  $v^{-\frac{1}{2}}e^{-\frac{1}{2}\frac{x^2}{v}}
  =
  u^{-\frac{1}{2}}e^{-\frac{1}{2}\frac{x^2}{u}}$,
  which implies that
  \begin{align*}
    \delta=\sqrt{\frac{\ln (\frac{v}{u})}{\frac{1}{u}-\frac{1}{v}}}
    =
    \sqrt{\frac{\ln (1+\frac{v-u}{u})}{\frac{1}{u}-\frac{1}{v}}}
    \le
    \sqrt{\frac{\frac{v-u}{u}}{\frac{v-u}{uv}}}
    = \sqrt{v}.
  \end{align*}
  Moreover,
  \begin{align*}
    \sqrt{\frac{v}{u}}
    =\sqrt{1+\frac{v-u}{u}}
    \le 1+\frac{1}{2}\frac{v-u}{u}.
  \end{align*}
  Putting the above bounds for $\delta$ and $\sqrt{v/u}$ back into \eqref{eq:app_tv} and using that $p_v(0)=(2\pi v)^{-\frac{1}{2}}$, we easily conclude. 
 \end{proof}

\begin{lemma}\label{lem:appendix_estimates}
  Let $B_t$, $t\ge 0$ be a BM and define $\tau_a:=\inf\{t\ge 0:B_t\ge a\}$ and $\tau_{a,a}:=\inf\{t\ge 0: B_t\not\in(-a,a)\}$; the corresponding density functions exist and we denote them by $p^a$ and $p^{a,a}$.  Then the following holds:
  \begin{itemize}
  \item[i)] For any $\beta>0$, there exists $\Delta>0$ such that for $\varepsilon<\Delta$, 
    \begin{align*}
      \P(\tau_a\le \varepsilon)\le \beta\varepsilon;
    \end{align*} 
  \item[ii)] For any $\beta>0$, there exists $\delta>0$ and $\Delta>0$, such that for $\varepsilon<\Delta$,  
    \begin{align*}
      \int_0^\delta | p^a(t+\varepsilon)-p^a(t) | \dt \le \beta \varepsilon.  
    \end{align*}
  \item[iii)] For any $\beta>0$, there exists $\delta>0$ and $\Delta>0$, such that for $\varepsilon<\Delta$,  
    \begin{align*}
      \int_0^\delta | p^{a,a}(t+\varepsilon)-p^{a,a}(t) | \dt \le \beta \varepsilon.  
    \end{align*}
  \end{itemize}
\end{lemma}

\begin{proof}
  i)	Recall that by the reflection principle 
  \begin{align*}
    \P(\tau_a\le t)=\P(|B_t|\ge a)=2\int_a^\infty \frac{1}{\sqrt{2\pi t}}e^{\frac{-x^2}{2t}}\dx;
  \end{align*} 
  differentiating twice in $t$ we obtain
  \begin{align}\label{eq:density_tau_a}
    p^a(t)=\frac{a}{\sqrt{2\pi t^3}}e^{\frac{-a^2}{2t}}. 
    \quad\textrm{and}\quad
    \pian{p^a}{t}(t)= \frac{a(a^2-3t)}{\sqrt{4\pi t^7}}e^{\frac{-a^2}{2t}}.
  \end{align}
  Next, since $t\mapsto \P(\tau_a\le t)$ is differentiable, we have $\P(\tau_a\le \varepsilon)=p^a(s)\varepsilon$ for some $s<\varepsilon$. We note from \eqref{eq:density_tau_a} that $p^a$ is continuous and that $\lim_{t\to 0}p^a(t)=0$; hence, for any $\beta>0$, by choosing $\Delta$ small enough, we may ensure that $p^a(s)\le \beta$, for all $s<\Delta$ which yields the claim.
			
  ii) We see from \eqref{eq:density_tau_a} that $p^a(t)$ is increasing for $t\le \bar t$ and decreasing for $t\ge \bar t$, where $\bar t=a^2/3$.  Let $\Delta< \bar t$ and $\delta:=\bar t - \Delta$. Then, for any $\varepsilon<\Delta$,
  \begin{align*}
    \int_0^\delta | p^a(t+\varepsilon)-p^a(t) | \dt
    \le 
    \P(\tau_a\le \delta+\varepsilon)-\P(\tau_a\le \delta)
    =p^a(s)\varepsilon, 
  \end{align*}
  for some $s\in (\delta,\bar t)$. By use of the same arguments as in i), for any $\beta>0$, choosing if necessary $\bar t$ even smaller, we may ensure that $p^a(s)\le \beta$ for all $s<\bar t$ which yields the claim. 
		
  iii) According to \cite{Borodin:aa} (see pp. 355 and 641),
  \begin{align*}
    p^{a,a}(t)
    =\sum_{-\infty}^\infty (-1)^k p^{a(1+2k)}(t)
    =2 \sum_{-\infty}^\infty p^{a(1+4k)}(t);
  \end{align*}
  hence $p^{a,a}$ is continuously differentiable and $\lim_{t\to 0}p^{a,a}(t)=0$. Since it is further a density function and thus non-negative, there must exist $\bar t>0$ such that $p^{a,a}(t)$ is non-decreasing for $t\le \bar t$. We may thus apply the same argument as in ii) to conclude. 
\end{proof}

\begingroup
\sloppy

\renewbibmacro*{in:}{}

\printbibliography

\endgroup


\end{document}
